\documentclass[reqno]{amsart}
\usepackage{graphicx,url,amsfonts,amsmath,amsthm,amssymb}
\usepackage{enumerate}
\usepackage[spanish,english]{babel}
\usepackage[utf8]{inputenc}
\makeatletter
\let\@fnsymbol\@arabic
\makeatother

\newtheorem{thm}{Theorem}[section]
\newtheorem{prop}{Proposition}[section]

\theoremstyle{definition}

\newtheorem{defi}{Definition}[section]
\newtheorem{exa}{Example}[section]

\def\Z{\mathbb{Z}}
\newcommand{\Ai}{\mathcal{A}}
\newcommand{\Fi}{\mathcal{F}}
\newcommand{\Li}{\mathcal{L}}
\newcommand{\Ri}{\mathcal{R}}
\newcommand{\Gb}{\mathbb{G}}
\newcommand{\Hb}{\mathbb{H}}
\newcommand{\Lb}{\mathbb{L}}
\newcommand{\Mb}{\mathbb{M}}
\newcommand{\domi}{dom}
\newcommand{\dom}{dom\,}
\begin{document}
\author[M. Aranguren]{Maira Aranguren$^{\MakeLowercase{1}}$}
\address{$^{\MakeLowercase{1}}$Universidad Nacional Experimental 
Polit\'ecnica Antonio Jos\'e de Sucre. 
Sección de Matem\'aticas. Av. Corpahuaico
Barquisimeto, Venezuela.}
\email{aranmarve@gmail.com}
\email{jorgkmpos@gmail.com}
\email{ramon.alberto.vivas@gmail.com}

\author[J. Campos]{Jorge Campos$^{\MakeLowercase{1}}$}

\author[N. Romero]{Neptal\'i Romero$^{\MakeLowercase{2}}$}
\address{$^{\MakeLowercase{2}}$Universidad Centroccidental
Lisandro Alvarado. Departamento
de Matem\'atica. Decanato de Ciencias y Tecnolog\'{\i}a.
Apartado Postal 400.
Barquisimeto, Venezuela.}
\email{neptali.romero@gmail.com (Corresponding author)}

\author[R. Vivas]{Ram\'on Vivas$^{\MakeLowercase{1}}$}

\title{Shift spaces, Languages and Transfinite induction}
\begin{abstract}
This paper deals with an extension of the classical concept of shift space,
which corresponds to any
shift-invariant closed subset of the Cartesian product of a particular finite set
(alphabet) 
endowed with the prodiscrete topology. 
In such an extended framework the notion of language is introduced and  
a characterization is shown. In order to do this, transfinite induction is required 
because the cardinality of the index set of the product may not be countable.
\end{abstract}

\maketitle
\section{Introduction}

In the classical context, a shift space is a special subset of a
particular product space. More precisely, let $\Ai$ be a finite 
nonempty set (an {\em alphabet}) endowed with the discrete topology.
On the Cartesian product $\Ai^\Z$ of all
doubly infinite sequences $x:\Z\to\Ai$ 
is considered the {\em prodiscrete topology}, that is the product topology
provided by the discrete topology of $\Ai$; this topological space
is known as the {\em full shift}.
The {\em shift mapping} $\sigma:\Ai^\Z\to\Ai^\Z$ maps a
sequence $x$ to the sequence $\sigma(x)$ whose $i$th coordinate 
$\sigma(x)_i$ is $x_{i+1}$
for every $i\in\Z$, this means that $\sigma(x)$ is obtained from $x$
by shifting each of its coordinates one place to the left. The mapping
$\sigma$ is continuous and bijective, its inverse $\sigma^{-1}$ shifts the
coordinates one place to the right; so $\sigma$ is a homeomorphism of 
$\Ai^\Z$. With these ingredients, a {\em shift space} (or {\em subshift})
of the full shift $\Ai^\Z$ is any nonempty closed subset $X$ of $\Ai^\Z$
satisfying the {\em shift-invariant} property, that is $\sigma(X)=X$. 

It is simple to obtain shift spaces of $\mathcal{A}^\mathbb{Z}$.
To this end, we first recall that a {\em word} (or {\em block}) 
over $\mathcal{A}$ is a finite sequence 
of symbols in $\mathcal{A}$, the number of symbols in it is called the 
{\em length} of the word. Every word over $\Ai$ is 
clearly identified with
an element of the Cartesian product $\Ai^k$ for some $k\geq 0$,
the {\em empty word}, usually denoted by $\epsilon$, lies in 
$\mathcal{A}^0$.
Let $\mathcal{A}^*=\bigcup_{k\geq 0}\mathcal{A}^k$ denote the set of 
blocks over $\mathcal{A}$, we say that a block 
$w\in\mathcal{A}^*$ {\em appears} in $x\in\mathcal{A}^{\mathbb{Z}}$
(in symbols $w\sqsubset x$)
if there exist $i,j\in\mathbb{Z}$ ($i\leq j$) such that 
$[i,j]=\{i,i+1,\cdots,j\}$ 
and the restriction $x\vert_{[i,j]}$ of $x$ to 
$[i,j]$ is just $w$; that is $x\vert_{[i,j]}=x_i\cdots x_j=w$. In this way
it is well known, see {\cite[Definition 2.1, Theorem 6.1.21]{lind}}, that for
any $\mathcal{F}\subset \mathcal{A}^*$ the set
$$
X_{\mathcal{F}}=\{x\in\mathcal{A}^\mathbb{Z}: 
\mbox{$x\vert_{[i,j]}\notin \mathcal{F}$ 
for every $i,j\in\Z$ with $i\leq j$}\}
$$
is a shift space. Note that if $\mathcal{F}=\emptyset$, then $X_{\mathcal{F}}$
is the full shift. The set $\mathcal{F}$ is called {\em a set of forbidden words} for the
shift space $X_{\mathcal{F}}$. We would like to emphasize that it is 
always possible to find
different nonempty sets $\mathcal{F}_1,\mathcal{F}_2$ of forbidden words 
for the same shift space; that is
$X_{\mathcal{F}_1}=X_{\mathcal{F}_2}$. 

Naturally a question arises:  
can any shift space be described as $X_{\mathcal{F}}$ for some 
$\mathcal{F}\subset\mathcal{A}^*$? The answer is yes and it is related to
a particular set of words. 
Given a shift space $X\subset \mathcal{A}^\mathbb{Z}$, a word
over $\mathcal{A}$ is called {\em allowed} for $X$ if it appears in some element of 
$X$. The set $\mathcal{L}(X)$ of all allowed words for the shift space 
$X$ is called the {\em language} of $X$, that is
$$
\mathcal{L}(X)=\{w\in\mathcal{A}^*: 
\mbox{there exists $x$ in $X$ such that $w\sqsubset x$}\}.
$$ 
Notice that $X\subset X_{\mathcal{L}(X)^c}$ where 
$\mathcal{L}(X)^c$ is the complementary set $\mathcal{A}^*\setminus \mathcal{L}(X)$ 
of $\mathcal{L}(X)$. Actually $X= X_{\mathcal{L}(X)^c}$, the reciprocal inclusion 
follows from the closedness property of $X$; so  
$\mathcal{L}(X)^c$ is the largest set of forbidden words for $X$.
Consequently shift spaces with the same language are identical.

The language of any shift space satisfies the {\em factorial} 
and {\em extendable} properties, these are respectively:

\begin{enumerate}[$L_1$)]
\item
If $w$ is a block in $\mathcal{L}(X)$ and $u$ is a subblock of $w$, 
then $u\in\mathcal{L}(X)$.
\item
If $w\in\mathcal{L}(X)$, then there are nonempty blocks $u,v\in\mathcal{L}(X)$ 
such that the concatenation block $uwv$ is also in $\mathcal{L}(X)$. 
\end{enumerate}

Certainly these two properties have important relevance: they characterize the 
language of shift spaces of the full shift $\mathcal{A}^\mathbb{Z}$, more exactly:

\begin{thm}[{\cite[Proposition 1.3.4]{lind}}]
\label{thm:t1}
If $\Li$ is a subset of $\Ai^*$ with properties $L_1$ and $L_2$, then there
exists a shift space $X\subset\Ai^\Z$ such that $\Li(X)=\Li$.
\end{thm}

We refer to the classical textbooks \cite{kitchens} and \cite{lind} for 
basic notions, technical tools and problems on the standard
theory of shift spaces; see also \cite{cecc} and its newly published companion 
book \cite{cecc2}. 
Although from a classical point of view shift spaces 
are considered over finite alphabets, in certain environments it is necessary to 
consider alphabets with infinitely many symbols, such is the situation, for example, when the 
thermodynamic formalism is developed for the so-called {\em countable state Markov shifts}, 
see \cite{kitchens} and \cite{sarig}. Shift spaces over alphabets with 
infinite symbols were treated by M. Gromov in his seminal work on endomorphisms 
of symbolic algebraic varieties and topological invariants of dynamical systems, 
see \cite{gromov} and \cite{gromov2}. 
Indeed, the notion of shift space over infinite alphabets and universes other than
$\mathbb{Z}$
has had different approaches and uses in multiple contexts, even various definitions 
have been proposed in correlation to different purposes; see for example
\cite{cecc}, \cite{goncalves2}, \cite{goncalves}, \cite{ott} and \cite{sobbo}.
These outstanding facts point to the interest of the 
study of extensions of shift spaces.

The main goal of this paper is to introduce and
discuss some basic properties of an extension of the
standard notion of shift space. Our proposal maintains closedness and
shift invariance as core properties of the extended notion that
we call {\em $\mathbb{H}$-subshift}. It
is presented on the product space $\mathcal{A}^\mathbb{G}$, 
where $\mathcal{A}$ (called alphabet) is a discrete topological space,
$\mathbb{G}$ is a group of infinite order and $\mathbb{H}$ is a subgroup
of $\mathbb{G}$;
the precise definition is in Section \ref{Sec2}, see Definition \ref{defhshift}. 
Indeed, this kind of subshifts is an extension of 
the notion of shift space introduced in \cite{cecc}, where some properties of those 
shift spaces and the cellular automata defined on them are described throughout 
different exercise sections of that textbook. 
After the introduction of the $\mathbb{H}$-subshift's concept our study
follow the standard route. First, in the extended setting
the notions of allowed and forbidden patterns
are introduced, they are $\mathcal{A}$-valued functions whose domains are
finite subsets of $\mathbb{G}$ 
generalizing the classical concepts of 
allowed and forbidden words. Then we introduce the concept of language for
$\mathbb{H}$-subshifts and we discuss its most basic properties, including
a version of Theorem \ref{thm:t1} when $\mathbb{G}$ is a countable
group. The arguments that we use to prove that version do not work when
$\mathbb{G}$ is uncountable. Although an appropriate order on $\mathbb{G}$
could be considered, uncountable cardinality imposes certain obstacles to
our arguments; among others, a closed
interval in $\mathbb{G}$ is not always a finite set, as it happens in $\mathbb{Z}$.
Thus, in order to state and prove an uncountable version of 
Theorem \ref{thm:t1}, we introduce in Definition \ref{extpatt} 
a more general notion of patterns: the extended patterns, 
they are $\mathcal{A}$-valued functions
defined on some subsets of $\mathbb{G}$ including the finite ones.
Finally, we prove that if $\mathcal{L}_\infty$ is a set of extended patterns 
satisfying certain properties, something like in Theorem \ref{thm:t1}, 
then there is an $\mathbb{H}$-subshift whose language is the set of functions 
in $\mathcal{L}_\infty$ having finite domain. This is precisely our last result 
in the article: Theorem \ref{principal} below.
\section{$\Hb$-shift spaces and their languages}\label{Sec2}

Let $\mathcal{A}$ be an alphabet: a discrete topological space
with at least two elements, and let $\Gb$ be a
group of infinite order with identity element $e$. On the 
Cartesian product $\Ai^\Gb$ (set of all functions from $\Gb$ to $\Ai$) the
prodiscrete topology is considered, it has as a base the set of all cylinders 
$$
C(P)=\{x\in\Ai^\Gb: x\vert_{\dom P}=P\},
$$ 
where $P$ is an
$\Ai$-valued function whose 
domain $dom\,P$ is a finite subset of $\Gb$ and $x\vert_{dom\,P}$ denotes the 
restriction of $x:\Gb\to\Ai$ to $dom\,P$. A function $P$ as above is called
{\em pattern} over the group $\Gb$ and alphabet $\Ai$ (or simply pattern).
The elements of $\Ai^\Gb$ are
called {\em configurations} and $\Gb$ is the {\em universe} of 
the {\em full shift} $\Ai^\Gb$. We recall that the space $\Ai^\Gb$ is Hausdorff,
perfect and totally disconnected; furthermore, $\Ai^\Gb$ is compact iff 
$\Ai$ is finite and it is metrizable iff $\Gb$ has countable cardinality.
In this case a distance $d$ compatible with the topology is defined in the following 
way. Let $\{V_n\}_{n\geq 0}$ be a sequence of finite subsets of $\Gb$ such that
$V_n\subsetneq V_{n+1}$ for all $n\geq 0$ and 
$\Gb=\bigcup_{n\geq 0}V_n$, so for configurations $x$ and $y$ in $\Ai^\Gb$
$$
d(x,y)=\begin{cases}
0, \text{ if $x=y$} \\
2^{-k}, \text{ if $x\neq y$ and $k=\max\{n\geq 0: x\vert_{V_n}=y\vert_{V_n}$\}}.
\end{cases}
$$
In any case ($\Gb$ countable or not), a neighborhood base for a configuration $x$ is
given by the family of cylinders $\{C(P_{x, \Lb}): \Lb\in [\Gb]^{<\infty}\}$
where $P_{x,\Lb}$ is the pattern $x\vert_{\Lb}$ and $[\Gb]^{<\infty}$ is
the collection of nonempty finite subsets of $\Gb$.

We note that the previous discussion about $\Ai^\Gb$ is
independent of any algebraic structure of $\Gb$, but the following is not.
For each $g\in\Gb$ it is defined 
the {\em $g$-shift} map $\sigma^g:\Ai^\Gb\to\Ai^\Gb$ by
\begin{equation}\label{gshift}
\sigma^g(x)(\ell)=x(g\ell)\,\text{ for all $x\in\Ai^\Gb$ and $\ell\in\Gb$},
\end{equation} 
where $g\ell$ means the (left)product of $g$ times $\ell$ in the
binary operation on $\Gb$. The following properties are direct:
$\sigma^e$ is the identity map, every $\sigma^g$ is continuous, bijective
with $(\sigma^g)^{-1}=\sigma^{g^{-1}}$, so every $g$-shift map is a homeomorphism;
additionally,  
$\sigma^g\circ\sigma^h=\sigma^{hg}$ for all $h,g\in\Gb$.

\begin{defi}\label{defhshift}
Given a subgroup $\mathbb{H}$ of a group $\mathbb{G}$,
a nonempty set $X\subset\Ai^\Gb$ is said to be an {\em $\Hb$-subshift}
(or {\em $\Hb$-shift space}) if it is a closed set in the prodiscrete topology 
of $\Ai^\Gb$ and it is $\Hb$-{\em shift invariant}; that is, 
$\sigma^g(X)=X$ for all $g\in\Hb$.
\end{defi}

We highlight that the $\{e\}$-subshifts are the closed subsets of  
$\Ai^\Gb$ and the $\Gb$-subshifts are the shift spaces introduced
in \cite{cecc}. In addition, $\Ai^\Gb$ is an $\mathbb{H}$-subshift 
for any subgroup $\mathbb{H}$ of $\mathbb{G}$.
We also observe that if $Y$ is an arbitrary set such that 
$\sigma^g(Y)\subset Y$ for all $g\in\mathbb{H}$, then the
group condition of $\mathbb{H}$ implies that $Y$ is $\mathbb{H}$-shift
invariant.

\begin{exa}\label{exa1}
Consider the alphabet $\Ai=\{0,1\}$ and the usual additive group $\Z$. The 
set $X\subset \Ai^\Z$ of all sequences $x=(x_n)_{n\in\Z}$
with $x_n=x_{n-1}$ for every even $n\in\Z$ is $2\Z$-subshift but is not 
a shift space.
\end{exa}

Now let us consider any $\Hb$-subshift $X\subset \Ai^\Gb$. Take any 
pattern $P:\Lb\to\Ai$ over $\Gb$ and the corresponding cylinder $C(P)$. 
If $C(P)\cap X\neq\emptyset$, then the
pattern $P$ is called {\em allowed}, otherwise it is a 
{\em forbidden pattern}. Clearly a pattern $P$ is allowed for $X$ iff there 
exists $x\in X$ such that $x\vert_{dom\,P}=P$; i.e. $P$ {\em appears} in $x$
($P\sqsubset x$).
The language $\Li(X)$ of the $\Hb$-subshift
$X$ is defined as the set of all allowed patterns. Observe that for 
every $\Lb\in [\Gb]^{<\infty}$ there are always allowed patterns with domain
$\Lb$; indeed, for every $x\in X$ and each $\Lb\in [\Gb]^{<\infty}$
the pattern $x\vert_\Lb$ is allowed.
It is also clear that $\Li(X)$
is the union of all $\Li_{\Lb}(X)$ where $\Lb$ runs over $[\Gb]^{<\infty}$ and
$\Li_\Lb(X)$ denotes the set of all allowed patterns for $X$ with domain
$\Lb$.

\begin{exa}\label{exa2}
Let $X$ be as in Example \ref{exa1}. The following patterns have 
same image set but a different nature in $X$, one of them is an
allowed pattern for $X$ and the other is not. Define
$P:\{0,1\}\to\Ai$ and $Q:\{-1,0\}\to \Ai$ by
$P(0)=1=Q(-1)$ and $P(1)=0=Q(0)$.
Observe that $P\sqsubset x$ where
$x\in\mathcal{A}^\mathbb{Z}$ is defined by
$x(n)=0$ for all $n\leq 0$ and $x(n)=1$ otherwise; also note that $x$ is an element of 
$X$ and so $P$ is an allowed pattern. Further, since for each $y\in X$
it holds that $y(-1)=y(0)$, $Q$ is a forbidden pattern for $X$.
This occurrence cannot happen in the classical setting of shift spaces.
\end{exa}

The following simple result proves, as for classic shift space, 
that the correspondence associating
to any $\Hb$-subshift its language is injective.

\begin{prop}
If $X,Y\subset\Ai^\Gb$ are both $\Hb$-subshifts, then $\Li(X)=\Li(Y)$ if and only if
$X=Y$.
\end{prop}
\begin{proof}
Suppose $\mathcal{L}(X)=\mathcal{L}(Y)$.
If there exists $x\in X\setminus Y$, then one can select an allowed
pattern for $X$, say $P$, such that $P$ appears in $x$ and $C(P)\cap Y=\emptyset$.
But $P$ also belongs to $\Li(Y)$, so 
$C(P)\cap Y\neq \emptyset$ and $X=Y$. It is clear that $X=Y$ implies 
$\mathcal{L}(X)=\mathcal{L}(Y)$.
\end{proof}

Next we list basic properties of the $\Hb$-subshift language; first,
we introduce a couple of definitions.

\begin{defi}\label{def1}
Let $X\subset\Ai^\Gb$ be an $\Hb$-subshift. Given $g\in\Hb$ and 
patterns $P$ and $Q$
over $\Gb$ and alphabet $\Ai$:
\begin{enumerate}[a)]
\item 
The {\em $g$-shifted pattern} $gP$ of the pattern $P$ is the pattern with domain
$g\Lb=\{g\ell:\ell\in \Lb\}$ ($\Lb= dom\,P$) and 
$gP(g\ell)=P(\ell)$ for all $\ell\in \Lb$.
\item
$Q$ is said to be a {\em subpattern} 
of $P$ if $P$ extends $Q$; that is, $dom\,Q\subset dom\,P$ and
$P\vert_{dom\,Q}=Q$, in symbols $P\subseteq Q$.
\end{enumerate}
\end{defi}

\begin{prop}\label{properties1}
Let $X\subset\Ai^\Gb$ be an $\Hb$-subshift. The language $\Li(X)$ of $X$ has the
following properties:
\begin{enumerate}[$L_1$)]
\item
Each subpattern of $P\in\Li(X)$ also belongs to $\Li(X)$.

\item
For every $Q\in\Li(X)$ and $g\in\Gb$ there exists 
$P\in\Li(X)$ such that $P\subseteq Q$ and $dom\,P=dom\,Q\cup \{g\}$ .

\item
For every $P\in\Li(X)$ and $g\in\Hb$, the pattern $gP\in\Li(X)$.
\end{enumerate}
\end{prop}
\begin{proof}
Properties $L_1$ (factorial property) and $L_2$ (extendable property) are immediate 
from the definition of allowed pattern. Actually the third one is also clear, 
just observe that for every $g\in\Hb$, a pattern $P$ appears in 
$x\in X$ if, and only if, the pattern $gP$ appears in 
$\sigma^{g^{-1}}(x)$.
\end{proof}

Property $L_3$ above is the {\em $\Hb$-shift invariance of patterns},
it implies that for every $g\in\Hb$ and $\Lb\in [\Gb]^{<\infty}$
the sets of allowed patterns $\Li_\Lb(X)$ and $\Li_{g\Lb}(X)$ have 
same cardinality. 
As in the classical context, see Theorem \ref{thm:t1} above, we would like to 
characterize the languages of shift spaces by sets of patterns with certain properties, 
more specifically with $L_1$, $L_2$ and $L_3$. However, we believe that this is 
not possible in the case that the group $\Gb$ is not countable. Below we  
add an additional property and use transfinite induction to obtain a
desired characterization.

Now we generalize three classical results in the $\Hb$-subshift setting. 
The first one characterizes this symbolic structure in terms of forbidden 
pattern sets; the second describes the compactness of 
$\Hb$-subshifts and the third one states a characterization of the language 
when $\Gb$ is countable. 

Given a collection $\Fi$ of patterns over $\Gb$ with alphabet $\Ai$ we
denote by $\domi_\Fi$ the set of all domains of patterns in $\Fi$.  

\begin{prop}
A subset $X$ of $\Ai^\Gb$ is an $\Hb$-subshift if, and only if, 
there exists  a collection $\Fi$ of patterns over $\Gb$ with alphabet $\Ai$
such that $X=X_\Fi^\mathbb{H}$, being that
\begin{equation}\label{forbi1}
X_\Fi^\mathbb{H}=\{x\in\Ai^\Gb: \text{for all $g\in\Hb$ and $\Lb\in\domi_\Fi$, 
$\sigma^g(x)\vert_\Lb\notin \Fi$}\}.
\end{equation}
\end{prop}
\begin{proof}
First we show that $X_\Fi^\mathbb{H}$ is an $\Hb$-subshift. 
To simplify the notation, we set $Y=X_{\mathcal{F}}^\mathbb{H}$.
Take $x\notin Y$, so
there are $g\in\Hb$ and $P\in\Fi$, say $P:\Lb\to\Ai$, such that 
$\sigma^g(x)\vert_\Lb=P$.
Hence the open set $\sigma^{g^{-1}}(C(P))$, that is the cylinder $C(gP)$,
contains $x$ and it is disjoint from 
$Y$; consequently $Y$ is a closed subset of $\Ai^\Gb$.
On the other hand, let us take arbitrary $x\in Y$ and $g\in\Hb$. Thus, 
whatever $\Lb\in\domi_\Fi$ and $\ell\in\Hb$ are, 
as $\sigma^\ell(\sigma^g(x))=\sigma^{g\ell}(x)$ and $g\ell\in\Hb$, 
it follows that
$\sigma^\ell(\sigma^g(x))\vert_\Lb\notin\Fi$ and so $\sigma^g(x)\in Y$.
Hence $\sigma^g(Y)\subset Y$ for all $g\in\Hb$; and since 
$\Hb$ is actually a group, $\sigma^g(Y)=Y$ for all $g\in\Hb$. Therefore $Y$ is
an $\Hb$-subshift.

Now suppose that $X$ is an $\mathbb{H}$-subshift. 
Let $\Fi$ be the set of all forbidden patterns for $X$; that is, $P\in\Fi$ iff 
$P\notin\Li(X)$. We will show that $X=X_\Fi^\mathbb{H}$. Since for each $x\in X$ and 
$g\in\Hb$ the configuration $\sigma^g(x)$ belongs to $X$, it follows that
for every $\Lb\in [\Gb]^{<\infty}$ the pattern $\sigma^g(x)\vert_\Lb$ is allowed; 
this implies the inclusion $X\subset X_\Fi^\mathbb{H}$. 
Now take $x\in X_\Fi^\mathbb{H}$, we claim that any pattern
$P$ with $P\sqsubset x$ is allowed for $X$. In fact, if $P:\Lb\to\Ai$ is
a forbidden pattern for $X$ with $x\in C(P)$, then $\Lb\in\domi_\Fi$ and
from \eqref{forbi1} it holds that $x\vert_\Lb\in\Li(X)$; but $x\vert_\Lb=P$. 
This proves the claim. As $x$ is a cluster point of the closed space $X$, 
we deduce that $x\in X$. Thus $X_\mathcal{F}^\mathbb{H}\subset X$,
and the proof is complete.
\end{proof}

\begin{prop}
An $\Hb$-subshift $X\subset\Ai^\Gb$ is compact if, and only if, 
$\Li_\Lb(X)$ is a finite set for every singleton subset $\Lb$ of $\Gb$.
\end{prop}
\begin{proof}
When $\Ai$ is finite there is nothing to prove, so assume
that $\Ai$ is not finite.
If $X$ is compact, the finiteness of $\Li_\Lb(X)$ follows from
the continuity of each natural projections from $X$ to $\Ai$;
recall that $\mathcal{A}$ is a discrete topological space. 
Conversely, assume that $\Li_\Lb(X)$ is finite for all 
singleton subset $\Lb$ of $\Gb$.
Define $\Ai_g=\{P(g): P\in\Li_{\{g\}}(X)\}$ for every $g\in\Gb$. Clearly
$\Ai_g$ is finite, the prodiscrete topology of 
$Y=\prod_{g\in\Gb}\Ai_g$ is compact and $X\subset Y$. 
Also note that if $X$ is a compact set in $Y$, then from the continuity of 
the inclusion map from $Y$ to $\mathcal{A}^\mathbb{G}$ the compactness of 
$X$ in $\mathcal{A}^\mathbb{G}$ follows. Now we will show that $X$ is a closed set
of $Y$, and so it is a compact set in $Y$.
Suppose that $X$ is not closed in $Y$; that is, there exists 
$x\in Y\setminus X$ such that for every cylinder $C(P)$ with $x\in C(P)$ 
it holds that $Y\cap C(P)\cap X\neq \emptyset$; therefore
$C(P)\cap X\neq \emptyset$ and $P\in\Li(X)$. In particular we have that
$x\vert_{\Lb}\in\Li(X)$ for all $\Lb\in [\Gb]^{<\infty}$. 
On the other hand, 
since $x\notin X$, there
are $g\in\Hb$ and $\Lb\in [\Gb]^{<\infty}$ such that 
$\sigma^g(x)\vert_{\Lb}\notin \Li(X)$, recall that $X=X_\Fi^\mathbb{H}$ 
where $\Fi$ is the set of all patterns outside $\mathcal{L}(X)$.
However, the pattern $P=x\vert_{g\Lb}$ is in $\Li(X)$ and therefore
$g^{-1}P=\sigma^g(x)\vert_{\Lb}$ is also in $\Li(X)$.
\end{proof}

We would like to emphasize that the above result does not mean 
that an $\mathbb{H}$-subshift of $\mathcal{A}^\mathbb{G}$ is a 
subset of $\mathcal{B}^\mathbb{G}$ for some finite subset 
$\mathcal{B}$ of $\mathcal{A}$, 
the following example is a case in point.
\begin{exa}\label{compactexample}
We consider as alphabet $\mathcal{A}$ the set of non-zero real numbers.
Let $\mathbb{G}$ be the general linear group
of degree $n\geq 2$ over $\mathbb{R}$, this is the set of all $n\times n$ 
invertible matrices over $\mathbb{R}$ endowed with the standard matrix
multiplication; as subgroup $\mathbb{H}$ of $\mathbb{G}$ we consider 
the special linear group: the set of all matrices in
$\mathbb{G}$ with determinant $1$. Now let $X$ be the singleton subset
of $\mathcal{A}^\mathbb{G}$ whose unique element is the
function $x:\mathbb{G}\to\mathcal{A}$ given by
$x(A)=det\,A$ (determinant of $A$) for all $A\in\mathbb{G}$. Note that $X$ is
a closed subset of $\mathcal{A}^\mathbb{G}$, and for arbitrary 
$A\in\mathbb{G}$ and $B\in\mathbb{H}$ it holds that
$$
\sigma^B(x)(A)=x(BA)=det\,BA=det\,A=x(A),
$$
that is $\sigma^B(X)=X$ for all $B\in\mathbb{H}$, thus $X$ is
an $\mathbb{H}$-subshift and every symbol of $\mathcal{A}$ appears in $X$.
Clearly $X$ is a compact subset of $\mathcal{A}^\mathbb{G}$.
\end{exa}

\begin{prop}\label{countable}
Let $\Gb$ be a countable group and $\Hb$ be a subgroup of $\Gb$. 
If $\Li$ is a collection of patterns over
$\Gb$ and alphabet $\Ai$ satisfying properties $L_1, L_2$ and $L_3$, then
there exists an $\Hb$-subshift such that $\Li(X)=\Li$.
\end{prop}
\begin{proof}
Let us assume that $\mathbb{G}=\{g_0,g_1,\cdots\}$. 
Let $\Fi$ be the set of all patterns over $\mathbb{G}$ and 
alphabet $\mathcal{A}$ outside $\mathcal{L}$ and 
let $X=X_\mathcal{F}^\mathbb{H}$, we will prove that
$\mathcal{L}(X)=\mathcal{L}$. Let $P:\mathbb{L}\to\mathcal{A}$ be a 
pattern in $\mathcal{L}(X)$ and $x$ be a configuration in $X$ 
such that $P\sqsubset x$. From definition of $X_\mathcal{F}^\mathbb{H}$
one obtains, in particular, that $x\vert_{\mathbb{L}}=P\notin\mathcal{F}$;
in other words $P\in\mathcal{L}$ and thus $\mathcal{L}(X)\subset\mathcal{L}$.
Note that this inclusion does not dependent from the cardinality of
$\mathbb{G}$.

Now we take $P:\mathbb{L}\to\mathcal{A}$ in $\mathcal{L}$. Let $n$ an
integer such that $\{g_0,g_1,\cdots,g_n\}$ contains $\mathbb{L}$.
From property $L_2)$ we select $P_n:\{g_0,g_1,\cdots,g_n\}\to\mathcal{A}$ 
in $\mathcal{L}$ so that $P\subseteq P_n$. Indeed, with a recursive 
process on the natural numbers, for each $k\geq 1$, we choose 
$P_{n+k}:\{g_0,g_1,\cdots,g_{n+k}\}\to\mathcal{A}$ in $\mathcal{L}$
such that $P_{n+j}\subseteq P_{n+k}$ for all
$0\leq j<k$. Next we define $x:\mathbb{G}\to\mathcal{A}$ by
$x(g)=P_{n+k}(g)$ if $g\in dom\,P_{n+k}$. Note that 
$P\sqsubset x$, so to show that $P\in\mathcal{L}(X)$ we will check
that $x\in X$. Take any $g\in\mathbb{G}$ and 
$\mathbb{M}\in dom_{\mathcal{F}}$. Since $g\mathbb{M}$ is a finite 
subset of $\mathbb{G}$ we take $k\geq 1$ large enough so that
$g\mathbb{M}\subset \{g_0,g_1,\cdots,g_{n+k}\}$. By definition,
$R=x\vert_{g\mathbb{M}}$ is a subpattern of $P_{n+k}$, therefore
$R\in\mathcal{L}$ (property $L_1)$). On the other hand, as 
$\sigma^g(x)\vert_{\mathbb{M}}$ is the $g^{-1}$-shifted pattern of
$R$, property $L_3)$ implies that 
$\sigma^g(x)\vert_{\mathbb{M}}\in \mathcal{L}$ and 
consequently $x\in X$.
\end{proof}

Observe that the preceding proposition is just a paraphrasing of
Theorem \ref{thm:t1}, to obtain a version of this result when the group
$\Gb$ is uncountable one
requires additional tools and hypothesis. To trace a route in that direction 
we begin by considering a broader class of patterns over $\Gb$ and 
alphabet $\Ai$. 

\begin{defi}\label{extpatt}
A function $P:\Lb\to\Ai$ 
is said to be an {\em extended pattern over $\Gb$ and alphabet
$\Ai$} ({\em extended pattern,} for short) if
$\mathbb{L}\subset\mathbb{G}$ and 
$\mathbb{G}\setminus\mathbb{L}$ is not finite.
\end{defi}

Note that every pattern over $\Gb$ and alphabet $\Ai$ is also
an extended one. If $X\subset \Ai^\Gb$ is an
$\Hb$-subshift, an extended pattern $P:\Lb\to\Ai$ is said to be 
{\em allowed} for $X$ if
there exists $x\in X$ such that $x\vert_{\Lb}=P$; as before one says
that $P$ appears at $x$ ($P\sqsubset x$). The set $\Li_\infty(X)$ 
of all allowed extended patterns for $X$ is the {\em extended language} of $X$.
Observe that $\Li(X)\subset \Li_\infty(X)$ and 
$P:\Lb\to\Ai \in\Li_\infty(X)$ iff $P\sqsubset x$ for some $x\in X$. 
For extended patterns, the notions of subpattern and $g$-shifted pattern (see
Definition \ref{def1}) are analogous.
It is also straightforward to see that the properties stated in 
Proposition \ref{properties1} also hold for the extended language. 

In order to reach our goal when
$\Gb$ is not countable, we need a well-order on $\Gb$ to use the principle
of transfinite induction. We recall some necessary facts and properties
about the notions of well-ordered sets and transfinite induction; for 
details see for example \cite{diprisco} or \cite{halmos}. 
Let $A$ be a nonempty set, a well-order on $A$ is a partial order,
say $<$, such that every nonempty subset of $A$ has a smallest element; 
in particular, $<$ is a total order (any
two elements are comparable) and each element $a$ in the well-ordered
set $(A,<)$, except a possible greatest one, has a unique immediate successor,
namely,
the smallest element of the subset $\{s\in A: a<s\}$. For every
$a\in A$, the {\em initial segment} determined by $a$ is the set
$s(a)=\{s\in A: s<a\}$; the corresponding {\em weak initial segment}
is $\overline{s}(a)=s(a)\cup\{a\}$. It is clear that if $a_0$ is the first element
of $A$, then $s(a_0)=\emptyset$; additionally, if $a_1$ is the first element
of $A\setminus\{a_0\}$, then $s(a_1)=\{a_0\}$ and so on. The collection
of all initial segments in $(A,<)$ is well-ordered by the inclusion $\subset$;
indeed, $s(a)\subset s(b)$ iff $a<b$. It is known that every well-ordered set
$(A,<)$ is identified (via a unique order isomorphism) with a ordinal number; 
moreover, the position of each element in $(A,<)$ is also given by a unique
ordinal number. So, every element of $(A,<)$ is either the zero, a successor
ordinal or a limit ordinal. In other words, given $a\in A$, it is either
$a_0$, it is a successor or for each $b\in A$ with $b<a$ there is $c\in A$ 
such that $b<c<a$.

Associated to well-order relations there is an important mathematical statement that
is equivalent to axiom of choice, that statement (also called Zermelo's theorem)
is the following:

\smallskip
\noindent
{\bf Well-ordering theorem.}
{\em Every nonempty set can be well-ordered.}

\smallskip

Some facts about well-ordered sets are well known (for details see 
for example \cite{diprisco} or \cite{enderton}):
\begin{enumerate}[a)]
\item 
Every well-ordered set $(A,<)$ is order isomorphic to a unique 
ordinal number $\alpha$; that is, $A$ is well-ordering in order type $\alpha$.
\item 
If $A$ is an infinite set and it is well-ordering in order type $|A|$ 
(the cardinality of $A$), then the complement of any proper initial 
segment in $A$ is not a finite set.
\end{enumerate}

The {\em principle of transfinite induction} is the 
statement that extends the principle of mathematical induction
to well orders of order types larger that the first infinite ordinal.
A version of it, written in the symbology of set theory, is:

\smallskip
\noindent
{\bf Principle of transfinite induction.}
{\em If $(A,<)$ is a well-ordered set and $\phi$ is a formula in the language of
set theory, then}
$$
(\forall u\in A[\forall v\in A(v<u \rightarrow \phi(v))\rightarrow \phi(u)])
\rightarrow \forall u\in A \phi(u).
$$

After that digression we return to our framework, it is constituted
by: an uncountable group $\mathbb{G}$ which is well-ordering in order type
$|\mathbb{G}|$, a subgroup $\mathbb{H}$ of $\Gb$, an alphabet 
$\mathcal{A}$ (a discrete topological space) and the Cartesian product
$\Ai^\Gb$ endowed with the 
prodiscrete topology. 
\begin{defi}\label{nested}
Given a set $\mathcal{L}_\infty$ of extended patterns over $\mathbb{G}$ 
and alphabet $\mathcal{A}$, an {\em increasing well-ordered $\subseteq$-chain} in
$\mathcal{L}_\infty$ is a map $\mathcal{R}:\Lambda\to\mathcal{L}_\infty$
such that: $(\Lambda,<)$ is a well-ordered set and for all 
$\alpha,\beta\in\Lambda$ with $\alpha<\beta$ it holds  
$\mathcal{R}(\alpha)\subseteq \mathcal{R}(\beta)$.
An increasing well-ordering $\subseteq$-chain $\mathcal{R}:\Lambda\to\mathcal{L}_\infty$
is called {\em proper} if 
$\mathbb{G}\setminus \bigcup_{\lambda\in\Lambda}
dom\,\mathcal{R}(\lambda)$ is infinite.
\end{defi}

Note that for every increasing well-ordered $\subseteq$-chain 
$\mathcal{R}:\Lambda\to\mathcal{L}_\infty$
its union
$$
\left(\bigcup\nolimits_{\alpha\in\Lambda}\Ri_\alpha\right)(g)=\Ri_\beta(g),
\,\text{if $g\in\dom \Ri_\beta$},
$$
makes sense; here $\mathcal{R}_\alpha=\mathcal{R}(\alpha)$, 
whatever $\alpha\in\Lambda$.  

As last ingredient to state and prove an uncountable version of
Theorem \ref{thm:t1}, we say that a set $\mathcal{L}_\infty$ of 
extended patterns satisfies the $L_4)$ property 
({\em invariance of increasing well-ordered $\subseteq$-chains}) if the union
of any proper increasing well-ordered $\subseteq$-chains in $\mathcal{L}_\infty$ also
belongs to $\mathcal{L}_\infty$.

\begin{thm}\label{principal}
Let $\mathcal{L}_\infty$ be a collection of extended patterns over 
$\mathbb{G}$ with alphabet $\mathcal{A}$ and satisfying each
property $L_i)$, $i=1,2,3,4$. If
$\mathcal{L}$ is the set of all patterns in $\mathcal{L}_\infty$
and $\mathbb{H}$ is a subgroup of $\mathbb{G}$,
then there exists an 
$\mathbb{H}$-subshift $X\subset\mathcal{A}^\mathbb{G}$ with 
$\mathcal{L}(X)=\mathcal{L}$.
\end{thm}
\begin{proof}
As in the proof of Proposition \ref{countable} we consider: the set
$\mathcal{F}$ of all patterns over $\mathbb{G}$ and alphabet
$\mathcal{A}$ outside $\mathcal{L}$ and $X=X^\mathbb{H}_\mathcal{F}$.
As in the case $\mathbb{G}$ countable we have $\mathcal{L}(X)\subset \mathcal{L}$,
see Proposition \ref{countable}.
Take an arbitrary $P:\mathbb{L}\to\mathcal{A}$ in $\mathcal{L}$,
proving that $P\in\mathcal{L}(X)$ is equivalent to show that there is 
$x\in X$ such that $P\sqsubset x$. In Proposition \ref{countable} 
we proceeded by extending $P$ by finite steps and used the principle of 
mathematical induction to obtain such a 
configuration $x$; this procedure does not work when 
$\mathbb{G}$ is uncountable. The way to extend $P$ begins by considering 
in $\mathbb{G}\setminus\mathbb{L}$ the well-order inherited from
the well-order $<$ considered in $\mathbb{G}$, recall that it is in order 
type $|\mathbb{G}|$. Let $g_0$ be the first element of 
$(\mathbb{G}\setminus \mathbb{L},<)$. For each $n\geq 1$ we denote by 
$g_n\in \mathbb{G}\setminus \mathbb{L}$ the immediate successor of
$g_{n-1}$ and by $g_\omega$ the first element in 
$(\mathbb{G}\setminus \mathbb{L},<)$ that is not successor of any other,
it is its first limit ordinal; or what is the same:
the first element of $\mathbb{G}\setminus 
(\mathbb{L}\cup \bigcup_{n\geq 1}s(g_{n}))$;
note that $s(g_\omega)=\{g_n: n\geq 0\}$.
As in the countable case, from property $L_2)$, it is obtained
a sequence $(P^{g_n})_{n\geq 0}$ of patterns
$P^{g_n}:\Lb\cup s(g_n)\to\Ai$ in $\mathcal{L}$ such that 
$P^{g_0}=P$ and $P^{g_{n-1}}\subseteq P^{g_n}$ for every $n\geq 1$.
The sequence $(P^{g_n})_{n\geq 0}$ is just a proper increasing
well-ordering $\subseteq$-chain in $\mathcal{L}_\infty$, so from property 
$L_4)$ its union 
$P^{g_\omega}:\Lb\cup s(g_\omega)\to\Ai$
is an element of $\mathcal{L}_\infty$.
Since $\mathbb{G}\setminus \mathbb{L}$ is uncountable, the above 
inductive process can be repeated up to the next countable ordinal 
and so on;
this is essentially the way to extend $P$ to obtain the desired configuration $x$. 
With a better accuracy, we state that an element $v$ in 
$(\mathbb{G}\setminus\mathbb{L},<)$
satisfies the formula $\phi$, that is $\phi(v)$, if
for all $w< v$ there exists 
$P^w:\mathbb{L}\cup s(w)\to\mathcal{A}$ in 
$\mathcal{L}_\infty$ such that $P\subseteq P^w$ and 
$P^z\subseteq P^w$ for each $z< w$. Symbolically the definition
of $\phi$ is:
\begin{equation}\label{formula}
\forall w< v \,\exists P^w:\mathbb{L}\cup s(w)\to\mathcal{A} 
\in\mathcal{L}_\infty (P\subseteq P^w \wedge \forall z< w 
(P^z\subseteq P^w)).
\end{equation}  

Clearly $\phi(g_n)$ for all $n\geq 0$. Now fix any $u\in\Gb\setminus\Lb$ 
and suppose $\phi(v)$ for each
$v< u$. For such a point $v$ we consider the increasing 
well-ordering $\subseteq$-chain $(P^w)_{w< v}$ with
$P^w: \Lb\cup s(w)\to\Ai$ as in \eqref{formula}, thus from 
property $L_4)$ the union $P^v=\bigcup\nolimits_{w< v}P^w$ 
belongs to $\Li_\infty$ and 
$P^w\subseteq P^v$ for all
$w< v$. Hence, it follows from the principle of transfinite induction 
that $\phi$ is satisfied  
for all $u\in\mathbb{G}\setminus\mathbb{L}$. Note that
for each $u\in\mathbb{G}\setminus \mathbb{L}$ one has 
an increasing well-ordering $\subseteq$-chain $(P^v)_{v< u}$
in $\mathcal{L}_\infty$ with $P^w\subseteq P^v$ for all
$w<v$; moreover, from property $L_4$ 
the union $P^u=\bigcup\nolimits_{v< u}P^v$ is in $\Li_\infty$. 
Also note that $(P^u)_{u\in\mathbb{G}\setminus\mathbb{L}}$ is 
an increasing well-ordering $\subseteq$-chain with
$\bigcup_{g\in\mathbb{G}\setminus\mathbb{L}}dom\,P^u=\Gb$. So
we define
$x:\Gb\to\Ai$ by
$$
x(u)=\begin{cases}
P(u),\text{ if $u\in\Lb$}\\
P^{u^+}(u),\text{ if $u\notin\Lb$}
\end{cases}
$$
where $u^+$ is the successor of $u$ in $(\Gb\setminus\Lb,<)$,
we will prove that $x\in X$.
Take any $g\in\Hb$ and $\Mb\in [\Gb]^{<\infty}$. Since the pattern
$\sigma^g(x)\vert_{\Mb}$ is the $g^{-1}$-shifted of $x\vert_{g\Mb}$,
if one shows that $x\vert_{g\Mb}$ is in $\Li$, then so is
$\sigma^g(x)\vert_{\Mb}$; consequently $x\in X$ and $P\in\Li(X)$.
First note that if
$g\Mb\subset\Lb$, then $x\vert_{g\Mb}\in\Li$; 
indeed, it is a subpattern of $P$.
Now, if $g\Mb\setminus\Lb=\{u_1, \cdots, u_m\}$ with
$u_1< \cdots < u_m$, then $x\vert_{g\Mb}$ is a subpattern of
$P^{u_m^+}$. In this way we conclude that
$x\vert_{g\Mb}$ is in $\Li$ and the proof of 
Theorem \ref{principal} is complete.
\end{proof}
\section{Conclusions}
In this paper we have introduced the concept of $\mathbb{H}$-subshift, it extends
the classical notion of shift space as shift-invariant closed subset of the full shift 
$\mathcal{A}^\mathbb{Z}$ with $\mathcal{A}$ any finite nonempty set. 
In our extended notion the alphabet $\mathcal{A}$ 
supporting that concept is any discrete topological space and a group $\Gb$ replaces
the standard additive  group $\mathbb{Z}$. We only deal with some very basic 
properties of these shift spaces, particularly those related to the notion of language 
that this symbolic structure generates. As core of this paper we describe 
some basic properties of these languages, we also present an extended notion of them 
and offer sufficient conditions to characterize them. To achieve this last goal we 
use transfinite induction since the group G can be uncountable. 
The concept of $\mathbb{H}$-subshift and properties treated in this
paper could be considered as part of a starting point for 
subsequent work. First, to explore properties such as 
those discussed for the extensions of shift space exposed in \cite{cecc} and 
\cite{sobbo}; on the other hand, to study the morphisms derived from 
the $\mathbb{H}$-subshift notion.

\bigskip
\noindent
\paragraph{\em Acknowledgements:} This work was partially supported by the Consejo de Desarrollo Científico, Humanístico y
Tecnológico (CDCHT) of the Universidad Centroccidental Lisandro Alvarado under grant 
1186-RCT-2019. 


\end{document}